\documentclass[10pt]{article}
\usepackage{amsmath, amsthm}\usepackage{enumerate}
\usepackage{amssymb}\usepackage{bbold}
\usepackage{color}
\newtheorem{theorem}{Theorem}[section]
\newtheorem{proposition}[theorem]{Proposition}
\newtheorem{definition}[theorem]{Definition}
\newtheorem{example}[theorem]{Example}
\newtheorem{lemma}[theorem]{Lemma}
\newtheorem{corollary}[theorem]{Corollary}

\DeclareMathOperator{\convo}{\xrightarrow[]{o}}

\DeclareMathOperator{\convr}{\xrightarrow[]{ru}}

\DeclareMathOperator{\convn}{\xrightarrow[]{\|\cdot\|}}

\large
\makeatletter
\renewcommand{\subsection}{\@startsection{subsection}{1}
{0pt}{3.25ex plus 1ex minus.2ex}{-1em}{\normalfont\normalsize\bf}}
\makeatother
\begin{document}

\title{{\bf Relatively uniformly continuous semigroups on ordered vector spaces}}

\author{Eduard Emelyanov$^{1}$, Nazife Erkur\c{s}un-\"Ozcan$^{2}$, Svetlana Gorokhova$^{3}$\\ 
\small $1$ Sobolev Institute of Mathematics, Novosibirsk, Russia\\
\small $2$ Hacettepe University, Ankara, Turkey\\
\small $3$ Southern Mathematical Institute -- the Affiliate of Vladikavkaz Scientific Center of RAN, Vladikavkaz, Russia}
\maketitle

\abstract{We study relatively uniformly continuous operator semigroups on ordered vector spaces and extend several recent results
obtained by M. Kramar Fijav\v{z}, M. Kandi\'{c}, M. Kaplin, and J. Gl\"{u}ck in the vector lattice setting to ordered vector spaces 
with generating cones.}


\large

\section{Introduction and preliminaries}

\hspace{4mm}
Relatively uniformly continuous semigroups on vector lattices (VLs) were introduced and studied by Kramar Fijav\v{z}, Kandi\'{c}, and Gl\"{u}ck
\cite{KK-F2020,KK2020,GK2024} in order to extend the theory of $C_0$-semigroups to the VL setting 
(see also a recent paper \cite{E2025} of the first author). It is rather natural and flexible to consider an abstract 
Cauchy problem on a dense subspace of a Banach space (see \cite{KK2020} and references therein). 
In a plenty of cases (e.g., a subspace of analytic functions), one may face even lack of the VL structure.
However, in various applications, underling PDEs describe processes in a compact area, and then
the space still possesses a structure of an ordered vector space (OVS) with a generating cone. 
This motivates us to study how far the results of papers \cite{GK2024, KK2020, KK-F2020}
can be extended to OVSs. 

Throughout the paper, vector spaces are real and operators are linear.
By $x_\alpha\downarrow 0$ we denote a decreasing net satisfying $\inf\limits_\alpha x_\alpha=0$. 
An OVS $X$ is said to have a {\em generating cone} if the span of its positive cone $X_+$ is $X$, or
equivalently $X_+-X_+=X$. An element $u\in X_+$ is said to be an {\em order unit} of $X$
if $X=\bigcup\limits_{r\in\mathbb{R}_+}r[-u,u]$. Clearly, if an OVS $X$ possesses an order unit
then the positive cone $X_+$ is generating. 

\begin{definition}\label{def of convergences in OVS}
A net $(x_\alpha)$ in an OVS $X$
\begin{enumerate}[$a)$]
\item\ 
{\em order converges} to $x\in X$ $($briefly, \text{\rm o}-converges to $x$ or $x_\alpha\convo x$$)$ 
if there exists a net $g_\beta\downarrow 0$ in $X$ such that, for each $\beta$, there is 
$\alpha_\beta$ with $\pm(x_\alpha-x)\le g_\beta$ for $\alpha\ge\alpha_\beta$. 
\item\ 
{\em relative uniform converges} to $x\in X$ $($briefly, \text{\rm ru}-converges to $x$$)$ 
if, for some $u\in X$ $($a {\em regulator} of the convergence$)$, there exists an 
increasing sequence $(\alpha_n)$ of indices 
satisfying $\pm(x_\alpha-x)\le\frac{1}{n}u$ for $\alpha\ge\alpha_n$.
We abbreviate this by $x_\alpha\convr x(u)$ or simply  $x_\alpha\convr x$.
\item\ 
{\em \text{\rm ru}-Cauchy} if the double net $(x_{\alpha'}-x_{\alpha''})$ \text{\rm ru}-converges to $0$.
\end{enumerate}
\end{definition}

\begin{definition}\label{def of subsets of OVS}
A subset $C$ of an OVS $X$ is called
\begin{enumerate}[$a)$]
\item\
\text{\rm ru}-{\em closed} if, for every net $(x_\alpha)$ in $C$,  $x_\alpha\convr x$ implies $x\in C$.
\item\
\text{\rm ru}-{\em dense} if, for each $x\in X$ there is a net $(x_\alpha)$ in $C$ such that $x_\alpha\convr x$. 
\item\
{\em \text{\rm ru}-{\em complete in}} $X$ if, for each \text{\rm ru}-Cauchy net $(x_\alpha)$ in $C$
with a regulator $u\in X$ there exist $x\in C$ and $w\in X$ such that $x_\alpha\convr x(w)$.
\end{enumerate}
\end{definition}

\noindent
As the \text{\rm ru}-convergence is sequential (see, e.g., \cite[Proposition 2.4]{AEG2021}), 
nets in Definition \ref{def of subsets of OVS} can be replaced by sequences.

\medskip
The complements to \text{\rm ru}-closed subsets of a VL $X$ forms the (not necessary linear) \text{\rm ru}-{\em topology} on $X$
introduced by W.A.L. Luxemburg and L.C-Jr. Moore in \cite{LM1967}.
It is well known that \text{\rm ru}-topology agrees with norm topology on a Banach lattices (BL)
(e.g., by \cite[Lemma 2.30]{AT2007}). By \cite[Example 2.2]{KK-F2020}, 
the \text{\rm ru}-topology on $L^p(\mathbb{R})$ is not locally convex when $0<p<1$.

\medskip
\noindent
It should be noted that the similarly defined \text{\rm ru}-{\em topology} on an OVS with a generating cone enjoy 
all the properties of \text{\rm ru}-{\em topology} on a VL listed in \cite[Section 4]{LM1967}  
with a little clarification that the use of modulus like $|g|\le f$ need to be replaced by $\pm g\le f$.

\medskip
In the present paper, we focus at $\mathbb{R}_+$-parameter operator semigroups.
Namely, under a {\em semigroup $(T_s)_{s\ge 0}$ on a vector space $X$} we understand  
\begin{enumerate}[-]
\item\
a family of operators on $X$ satisfying $T_{s+t}=T_sT_t$ for all $s,t\in\mathbb{R}_+$ and $T_0=I_X$, where $I_X$ is the identity operator on $X$.
\end{enumerate}
\medskip
\noindent
For unexplained terminology, notation, and basic results on OVSs, VLs, 
and operator semigroups, we refer to \cite{AT2007,E2007,EN2000,LZ1971,AB2006,V1967}.

\section{\large Relatively uniformly continuous at zero semigroups on ordered vector spaces}

\hspace{4mm}
In order to deal with partial differential equations, whose boundary
conditions are facing lack of the common topological structure and 
belong to just an Archimedean VL, Kandi\'{c} and Kaplin introduced 
an important concept of a relatively uniformly continuous semigroup on
a VL \cite[Definition 3.3]{KK2020}. We extend their definition as follows.

\begin{definition}\label{def of r-cont semigroup}
A semigroup $(T_s)_{s\ge 0}$ on an OVS $X$ is
\begin{enumerate}[$a)$]
\item 
{\em relatively uniformly continuous} $($or, a \text{\rm ruc} semigroup$)$ if $T_{s+h}x\convr T_sx$ for every $x\in X$ and $s\ge 0$ as $h\to 0$.
\item
{\em \text{\rm ruc} at zero}, if $T_{h}x\convr x$ as $h\downarrow 0$ for each $x\in X$.
\item 
{\em order continuous} if  $T_{s+h}x\convo T_sx$ for every $x\in X$ and $s\ge 0$, as $h\to 0$.
\item
{\em order continuous at zero} if $T_{h}x\convo x$ as $h\downarrow 0$ for each $x\in X$.
\item
{\em order bounded at zero} if, for each $x\in X$ there exists $\varepsilon_x>0$ such that 
the set $\{T_{h}x\}_{h\le\varepsilon_x}$ is order bounded.
\end{enumerate}
\end{definition}

\noindent
It is worth noting that in Archimedean OVSs \text{\rm ru}-convergence imply \text{\rm o}-convergence, and hence each \text{\rm ruc} (at zero) semigroup
on an Archimedean OVS is order continuous (at zero). There are two relatively easy cases when a semigroup 
admits a localization of the property to be \text{\rm ruc}/order continuous.

\begin{theorem}\label{localization at zero}
Let $(T_s)_{s\ge 0}$ be a semigroup on an OVS $X$ with a generating cone.
\begin{enumerate}[$i)$]
\item 
Suppose $(T_s)_{s\ge 0}$ is \text{\rm ruc} at zero and there exists $\varepsilon>0$ so that, for each $b\in X_+$,
there exists $v\in X_+$ such that $\bigcup\limits_{0\le s\le \varepsilon}T_s[-b,b]\subseteq[-v,v]$. Then $(T_s)_{s\ge 0}$ is \text{\rm ruc}.
\item
Suppose $(T_s)_{s\ge 0}$ is order continuous at zero and there exists $\varepsilon>0$ so that, for each $x_\alpha\convo 0$
there is a net $g_\beta\downarrow 0$ such that, for each $\beta$ there exists $\alpha_\beta$ satisfying 
$\pm T_sx_\alpha\le g_\beta$ for $\alpha\ge\alpha_\beta$ and $0\le s\le\varepsilon$. Then $(T_s)_{s\ge 0}$ is order continuous.
\end{enumerate}
\end{theorem}

\begin{proof}
$i)$\
Take arbitrary $x\in X$ and $s_0>0$, and let $h\downarrow 0$. By the assumption, $(T_hx-x)\convr 0$. 
By \cite[Proposition 3.3 and Lemma 3.2]{EEG2024}, for every $x_\alpha\convr 0$ there exist $u\in X_+$ and an increasing sequence 
$(\alpha_n)$ of indices with $\pm T_sx_\alpha\le\frac{1}{n}u$ for
$\alpha\ge\alpha_n$ and $0\le s\le s_0$. In particular, $T_{s_0}$ is \text{\rm ru}-continuous. Thus, $T_{s_0+h}x-T_{s_0}x=T_{s_0}(T_hx-x)\convr 0$,
and hence the function $s\to T_sx$ is right \text{\rm ru}-continuous at $s_0$.

Since $(x-T_hx)\convr 0$, there exist $w\in X_+$ and a decreasing sequence $(h_n)$ in $(0,s_0]$
with $\pm T_s(x-T_hx)\le\frac{1}{n}w$ for $0<h\le h_n$ and $0\le s\le s_0$.
In particular, $T_{s_0-h}x-T_{s_0}x=T_{s_0-h}(x-T_hx)\convr 0$ for $s_0\ge h\downarrow 0$, 
and hence the function $s\to T_sx$ is also left \text{\rm ru}-continuous at $s_0$.

\medskip
$ii)$\
Take $x\in X$ and $s_0>0$, and let $h\downarrow 0$. Then $(T_hx-x)\convo 0$.
By \cite[Proposition 3.3]{EEG2024}, for each $x_\alpha\convo 0$, there is a net $g_\beta\downarrow 0$ such that, 
for each $\beta$, there is $\alpha_\beta$ satisfying $\pm T_sx_\alpha\le g_\beta$ for $\alpha\ge\alpha_\beta$ and $0\le s\le s_0+1$.
In particular, $T_{s_0}$ is order continuous. Therefore, we obtain $T_{s_0+h}x-T_{s_0}x=T_{s_0}(T_hx-x)\convo 0$,
and hence $s\to T_sx$ is right order continuous at $s_0$.

Applying once more \cite[Proposition 3.3]{EEG2024}, find a net $q_\beta\downarrow 0$ such that, for each $\beta$, 
there is $h_\beta\in(0,s_0]$ with $\pm T_s(x-T_hx)\le q_\beta$ for $0<h\le h_\beta$ and $0\le s\le s_0$.
In particular, $T_{s_0-h}x-T_{s_0}x=T_{s_0-h}(x-T_hx)\convo 0$ for $s_0\ge h\downarrow 0$, 
and hence $s\to T_sx$ is also left order continuous at $s_0$.
\end{proof}

%
%

\medskip
$C_0$-semigroups need not to be \text{\rm ruc} at zero \cite[Example 3.12]{KK2020}. 
In the reverse direction we have the following partial result.

\begin{proposition}\label{ruc and oc to C_0}
$(T_s)_{s\ge 0}$ is a $C_0$-semigroup on $X$ if one of the following conditions holds.
\begin{enumerate}[$i)$]
\item
$X$ is an OBS with a closed generating normal cone, $(T_s)_{s\ge 0}$ 
is \text{\rm ruc} at zero, and there exists $\varepsilon>0$ such that, for each $b\in X_+$ 
the set $\bigcup\limits_{0\le h\le\varepsilon}T_h[0,b]$ is order bounded.
\item
$X$ is a BL with order continuous norm and $(T_s)_{s\ge 0}$ is an order continuous 
at zero semigroup which consists of bounded operators.
\end{enumerate}
\end{proposition}

\begin{proof}
$i)$\
It follows from \cite[Proposition 3.3(iv)]{EEG2024} that $\sup\limits_{0\le h\le M}\|T_h\|<\infty$ for every $M\in\mathbb{R}_+$.
Let $x\in X$ and $h\downarrow 0$. Then $T_hx\convr x$ because $(T_s)_{s\ge 0}$ is \text{\rm ruc} at zero. Normality of $X_+$ ensures 
$T_hx\convr x\Longrightarrow\|T_hx-x\|\to 0$.

\medskip
$ii)$\
Let $x\in X$ and $h\downarrow 0$.
Since the norm in $E$ is order continuous, the rest follows from $T_hx\convo x\Longrightarrow\|T_hx-x\|\to 0$ for all $x\in X$.
\end{proof}

\begin{proposition}\label{localization at zero of obz in OVS}
Let $(T_s)_{s\ge 0}$ be a positive semigroup on an OVS $X$ with a generating cone. The following assertions are equivalent.
\begin{enumerate}[$i)$]
\item
$(T_s)_{s\ge 0}$ is order bounded at zero.
\item
For every $x\in X$ and $M\ge 0$ the set $\{T_{h}x\}_{0\le h\le M}$ is order bounded.
\item
For every $x\in X_+$ and $M\ge 0$ the set $\bigcup\limits_{0\le s\le M}T_s[-x,x]$ is order bounded. 
\item
There exists $\varepsilon>0$ such that the set $\bigcup\limits_{0\le s\le\varepsilon}T_s[-x,x]$ is order bounded for every $x\in X_+$. 
\end{enumerate}
\end{proposition}

\begin{proof}
$i)\Longrightarrow ii)$\  
Let $x\in X$ and $M\ge 0$. By the assumption, there exist $a,b\in X$ and $\varepsilon>0$
satisfying $\big\{T_hx\big\}_{0\le h\le \varepsilon}\subseteq[a,b]$. Since $X_+-X_+=X$, $\{T_{h}x\}_{0\le h\le\varepsilon}\subseteq[-w,w]$
for some $w\in X$. Then 
$$
   \big\{T_hx\big\}_{0\le h\le 2\varepsilon}=\big\{T_hx\big\}_{0\le h\le \varepsilon}\bigcup T_\varepsilon\big(\big\{T_hx\big\}_{0\le h\le\varepsilon}\big)
   \subseteq[-w-T_\varepsilon w,w+T_\varepsilon w],
$$
and by repeating, for $n$ such that $n\varepsilon\ge M$, we have
$$
   \big\{T_hx\big\}_{0\le h\le M}\subseteq\big\{T_hx\big\}_{0\le h\le n\varepsilon}\subseteq
   \left[-\sum_{k=0}^{n-1}T_{k\varepsilon}w,\sum_{k=0}^{n-1}T_{k\varepsilon}w\right].
$$ 

\medskip
$ii)\Longrightarrow iii)$\  
Let $x\in X_+$ and $M\ge 0$. 
Since $X_+-X_+=X$, $ii)$ implies $\{T_sx\}_{0\le s\le M}\subseteq[-b,b]$ for some $b\in X$. 
Since $(T_s)_{s\ge 0}$ is positive then $-b\le -T_sx\le T_sy\le T_sx\le b$ 
for all $y\in[-x,x]$ and $s\in[0,M]$, and hence $\bigcup\limits_{0\le s\le M}T_s[-x,x]\subseteq[-b,b]$.

\medskip
$iii)\Longrightarrow iv)$\ 
It is trivial.

\medskip
$iv)\Longrightarrow i)$\ 
Let $x\in X$. Find $\varepsilon>0$ such that $\bigcup\limits_{0\le h\le\varepsilon}T_h[-z,z]$ is order bounded for every $z\in X_+$. 
Since $X_+$ is generating, $x\in[-z,z]$ for some $z\in X$, and hence the set $\{T_hx\big\}_{0\le h\le \varepsilon}$ is order bounded.
Since $x\in X$ is arbitrary, $(T_s)_{s\ge 0}$ is order bounded at zero.
\end{proof}

It is proved in \cite[Proposition 3.4]{KK2020} that, for each \text{\rm ruc} at zero positive semigroup $(T_s)_{s\ge 0}$ on a VL $X$,
the set $\{T_sx: 0\le s\le M\}$ is order bounded for every $M\ge 0$ and $x\in X$. 
By Proposition \ref{localization at zero of obz in OVS}, the same conclusion 
holds for each order bounded at zero positive semigroup on an OVS with a generating cone. 
In general, the assumption that the semigroup $(T_s)_{s\ge 0}$ is positive is essential in Proposition \ref{localization at zero of obz in OVS}.
The next example shows that even a uniformly continuous \text{\rm ruc} semigroup $(S_s)_{s\ge 0}$
on an order continuous Banach lattice $X$ can be order bounded at zero yet  $\bigcup\limits_{0\le s\le\varepsilon}S_s[0,x]$ 
is not order bounded for some $x\in X_+$ and each $\varepsilon>0$.

\begin{example}\label{example -(n+1)}
Let $E=\big(\bigoplus_{k=1}^\infty\ell^{2^n}\big)_0$ be the order continuous Banach lattice from the Krengel example 
{\em \cite[Example 5.6]{AB2006}}. Consider the operator $T:E\to E$ defined by $T(x_1,x_2,\cdots)=({T_1}x_1,{T_2}x_2,\cdots)$,
where $T_n:\ell^{2^n}\to\ell^{2^n}$ is invertable isometry from {\em \cite[Example 5.6]{AB2006}}.
Clearly, $\|T\|=1$. However, $T$ is not order bounded. Indeed, otherwise $|T|$ would exists because $E$ is Dedekind complete, 
and then $|T|{\bf x}=(|{T_1}|x_1,|{T_2}|x_2,\cdots)$ for every ${\bf x}=(x_1,x_2,\cdots)\in E$. Take 
$$
   x_n:=\frac{1}{n}\cdot 2^{-\frac{n}{2}}\sum_{i=1}^{2^n}e_i\in\ell^{2^n} \ \ \ \ \text{\rm and \ let} \ \ {\bf x}=(x_1,x_2,\cdots)\in E.
  \eqno(1)
$$
Denote by $M_n$ an $2^n\times 2^n$ matrix with all of its entries equal to $1$.
This leads to absurd, since  
$$
   |T|{\bf x}=(|{T_n}|x_n)_{n\in\mathbb{N}}=
   \left(\frac{1}{n}\cdot 2^{-\frac{n}{2}}|{T_n}|\left(\sum_{i=1}^{2^n}e_i\right)\right)_{n\in\mathbb{N}}=
$$
$$
   \left(\frac{1}{n}\cdot 2^{-\frac{n}{2}}\cdot 2^{-\frac{n}{2}} M_n\left(\sum_{i=1}^{2^n}e_i\right)\right)_{n\in\mathbb{N}}=
   \left(\frac{1}{n}\left(\sum_{i=1}^{2^n}e_i\right)\right)_{n\in\mathbb{N}}\notin E,
   \eqno(2)
$$
because $\left\|\frac{1}{n}\left(\sum_{i=1}^{2^n}e_i\right)\right\|_2=\frac{\sqrt{2^n}}{n}\to\infty$.

\medskip
Consider the order continuous Banach lattice $X=E\oplus E$, define a contraction $A:X\to X$ by $A({\bf f},{\bf g})=(0,T{\bf f})$,
and set $S_s=e^{sA}=\sum_{l=0}^{\infty}\frac{s^lA^l}{l!}$ for $s\ge 0$. $(S_s)_{s\ge 0}$ is a uniformly continuous semigroup,
and $S_s=I_X+sA$ for $s\ge 0$ because $A^2=0$.
Then $(S_s)_{s\ge 0}$ is \text{\rm ruc} since, for $({\bf f},{\bf g})\in X=E\oplus E$ and $s_0\ge 0$,
$$
   S_{s_0+h}({\bf f},{\bf g})-S_{s_0}({\bf f},{\bf g})=(I_X+(s_0+h)A-(I_X+s_0A))({\bf f},{\bf g})=
$$ 
$$
   hA({\bf f},{\bf g})=h(0,T{\bf f})\convr 0
$$ 
with regulator $(0,T{\bf f})$ as $s_0+h\ge 0$ and $h\to 0$. As \text{\rm ru}-converging nets are eventually order bounded,
$(S_s)_{s\ge 0}$ is also order bounded at zero.

\medskip
Note that the set $\bigcup\limits_{0\le s\le\varepsilon}S_s[0,({\bf x},0)]$ is not order bounded in $X$ for every $\varepsilon>0$,
where ${\bf x}\in X$ is as in $(1)$. Otherwise, the set $\bigcup\limits_{0\le s\le\varepsilon}sA[0,({\bf x},0)]$
would be order bounded in $E$ for some $\varepsilon>0$, and hence $T[-{\bf x},{\bf x}]$ must be order bounded in $E$. Then 
$(|{T_n}|x_n)_{n\in\mathbb{N}}\in({T_n}[-x_n,x_n])_{n\in\mathbb{N}}\subseteq T[-{\bf x},{\bf x}]$ which is absurd because
$(|{T_n}|x_n)_{n\in\mathbb{N}}\notin E$ by $(2)$.
\end{example}

\medskip
The following theorem generalizes \cite[Proposition 3.4]{KK2020} by dropping positivity of the semigroup in favor of the condition 
that the semigroup contains infinitely many order bounded operators in every neighborhood of zero.

\begin{theorem}\label{ob-ruc}
Let $(T_s)_{s\ge 0}$ be a semigroup on an OVS $X$ with a generating cone such that,
for some sequence $r_k\downarrow 0$ in $\mathbb{R}_+\setminus\{0\}$ every $T_{r_k}$ is order bounded. 
Then the set $\{T_sx: 0\le s\le M\}$ is order bounded for every $M\ge 0$ and $x\in X$ if one of the following two conditions holds.
\begin{enumerate}[$i)$]
\item
$(T_s)_{s\ge 0}$ is \text{\rm ruc} at zero.
\item
$(T_s)_{s\ge 0}$ is order continuous at zero.
\end{enumerate}
\end{theorem}

\begin{proof}
Take arbitrary $M\ge 0$ and $x\in X$.

\medskip
$i)$\
Since $T_hx\convr x$ as $h\downarrow 0$, there exist $u\in X$ and $\varepsilon>0$ satisfying $\pm(T_hx-x)\le u$ for all $h\in[0,\varepsilon]$. 
Pick any $r_{k_\varepsilon}\in(0,\varepsilon)$ and take $n$ with $nr_{k_\varepsilon}\ge M$.
Since $X_+-X_+=X$ then $\pm x\le w$ for some $w\in X_+$.
Thus, $\pm T_hx\le a_1:=u+w$ for $0\le h\le\varepsilon$, and hence $\big\{T_hx\big\}_{0\le h\le r_{k_\varepsilon}}\subseteq[-a_1,a_1]$.
Since $T_{r_{k_\varepsilon}}$ is order bounded and $X_+$ is generating,  
$$
   \{T_hx: 0\le h\le 2r_{k_\varepsilon}\}=T_{r_{k_\varepsilon}}\Big(\big\{T_hx\big\}_{0\le h\le r_{k_\varepsilon}}\Big)\subseteq 
   T_{r_{k_\varepsilon}}[-a_1,a_1]\subseteq[-a_2,a_2]
$$
for some $a_2\in X_+$. It follows $\big\{T_hx\big\}_{0\le h\le{r_{k_\varepsilon}}}\subseteq[-(a_1+a_2),a_1+a_2]$.
Repeating n-times the argument, we obtain $a_1,a_2,...,a_n\in X_+$ with 
$$
   \{T_hx: 0\le h\le M\}\subseteq\big\{T_hx\big\}_{0\le h\le n{r_{k_\varepsilon}}}\subseteq
   \left[-\sum_{i=1}^{n}a_i,\sum_{i=1}^{n}a_i\right].
$$

\medskip
$ii)$\
Take $M\ge 0$ and $x\in X$. Since $T_hx\convo x$ as $h\downarrow 0$, there exists a net $g_\beta\downarrow 0$ in $X$,
such that for each $\beta$ there is $k_\beta$ with $\pm(T_hx-x)\le g_\beta$ for $0\le h\le r_{k_\beta}$. 
Fix $\beta_0$ and take $w\ge\pm x$. Then $\pm T_hx\le a_1:=g_{\beta_0}+w$ for $h\in[0,r_{k_{\beta_0}}]$, and hence
$\big\{T_hx\big\}_{0\le s\le r_{k_{\beta_0}}}\subseteq[-a_1,a_1]$.
Let $nr_{k_{\beta_0}}\ge M$.  Taking in account that $T_{r_{k_{\beta_0}}}$ is order bounded and arguing as in the proof of $i)$, 
we obtain that $\big\{T_hx\big\}_{0\le h\le M}\subseteq\left[-\sum_{i=1}^{n}a_i,\sum_{i=1}^{n}a_i\right]$
for some $a_1,a_2,...,a_n\in X_+$. 
\end{proof}

\noindent
It should be clear that the assumption in Theorem \ref{ob-ruc} that operators $T_s$  are order bounded along a decreasing to zero sequence of
non-zero indices is equivalent to the assumption that $T_s$ is order bounded for $s$ belonging to a dense subset of $\mathbb{R}_+$.
Also, remark that if $X$ is Archimedean then condition $i)$ of Theorem \ref{ob-ruc} implies condition $ii)$.

\medskip
In the absence of a collective order boundedness (resp., collective order continuity) conditions like in 
Theorem \ref{localization at zero}, localization of the property of a semigroup to be \text{\rm ruc} (resp., order continuous) 
is a rather difficult task. The following theorem gives some partial results in this direction. 
Later, in Theorem \ref{Thm.5.7-KK2020} we give another localization result for the \text{\rm ruc}-property of a positive semigroup.

\begin{theorem}\label{loc-ruc}
Let $(T_s)_{s\ge 0}$ be a semigroup on an OVS $X$ with a generating cone. 
\begin{enumerate}[$i)$]
\item
If $(T_s)_{s\ge 0}$ is \text{\rm ruc} at zero and consists of order bounded operators then $(T_s)_{s\ge 0}$ is right \text{\rm ru}-continuous.
Furthermore, $(T_s)_{s\ge 0}$ is \text{\rm ruc} whenever it is positive.
\item
If $(T_s)_{s\ge 0}$ is order continuous at zero and consists of order continuous positive operators then $(T_s)_{s\ge 0}$ is right order continuous.
\end{enumerate}
\end{theorem}

\begin{proof}
$i)$\
Let $s_0\ge 0$, $h\downarrow 0$, and $x\in X$. Then $\big\{T_sx\big\}_{0\le s\le s_0+1}\subseteq[-u,u]$ for some $u\in X$
by Theorem \ref{ob-ruc}. The operator $T_{s_0}$ is \text{\rm ru}-continuous by \cite[Lemma 3.2]{EEG2024}, 
and hence $T_{s_0+h}x-T_{s_0}x=T_{s_0}(T_hx-x)\convr 0$.
Therefore, the mapping $s\to T_sx$ is right \text{\rm ru}-continuous at $s_0$.

Suppose additionally that $(T_s)_{s\ge 0}$ is positive. 
Since $(x-T_hx)\convr 0$  then $\pm(x-T_hx)\le\frac{1}{n}v$ for $0\le h\le h_n$
for some $v\in X$ and a sequence $h_n\downarrow 0$ in $[0,s_0]$. 
Then, $\pm T_s(x-T_hx)\le\frac{1}{n}T_sv$ for all $0\le h\le h_n$ and $s\ge 0$.
Theorem \ref{ob-ruc} implies $\big\{T_sv\big\}_{0\le s\le s_0}\subseteq[-w,w]$ for some $w\in X$. Consequently,
$\pm T_s(x-T_hx)\le\frac{1}{n}T_sv\le\frac{1}{n}w$ for all $0\le h\le h_n$ and $0\le s\le s_0$.
In particular, $\pm(T_{s_0-h}x-T_{s_0}x)=\pm T_{s_0-h}(x-T_hx)\le\frac{1}{n}w$ for $0\le h\le h_n$.
So, $T_{s_0-h}x\convr T_{s_0}x(w)$, and hence $s\to T_sx$ is left \text{\rm ru}-continuous at $s_0$.
As positive operators are order bounded, it was shown already that $s\to T_sx$ is right \text{\rm ru}-continuous at $s_0$.
Thus, $(T_s)_{s\ge 0}$ is \text{\rm ruc}.

\medskip
$ii)$\
Let $s_0\ge 0$, $h\downarrow 0$ and $x\in X$. Then $T_hx\convo x$. Find $g_\beta\downarrow 0$ in $X$ and,
for every $\beta$ some $\varepsilon_\beta>0$ such that $\pm(T_hx-x)\le g_\beta$ for all $h\in[0,\varepsilon_\beta]$. 
Since $T_{s_0}$ is positive, 
$$
   \pm(T_{s_0+h}x-T_{s_0}x)=\pm T_{s_0}(T_hx-x)\le T_{s_0}g_\beta
$$
for $h\in[0,\varepsilon_\beta]$.
Since $T_{s_0}$ is order continuous then $T_{s_0}g_\beta\downarrow 0$, 
and hence the mapping $s\to T_sx$ is right order continuous at $s_0$.
\end{proof}

\noindent
The next consequence of Theorem \ref{loc-ruc} generalizes \cite[Proposition 3.5]{KK2020} to the OVS setting.

\begin{corollary}\label{loc-ruc-gen}
Let $(T_s)_{s\ge 0}$ be a positive semigroup an OVS $X$ with a generating cone such that $T_{h}x\convr x$ as $h\downarrow 0$ for each $x\in X_+$.
Then $(T_s)_{s\ge 0}$ is \text{\rm ruc}.
\end{corollary}

\begin{proof}
By Theorem \ref{loc-ruc}, it suffices to show that $(T_s)_{s\ge 0}$ is \text{\rm ruc} at zero.
Assume $x\in X$ and $h\downarrow 0$. Then $x=y-z$ with $y,z\in X_+$. It follows from $T_hy\convr y$ and $T_hz\convr z$ 
that $T_hx=T_hy-T_hz\convr y-z=x$.
\end{proof}

\section{\large Localization at zero of the property to be ruc for positive semigroups on ordered vector spaces}

\hspace{4mm}
We shall use the following OVS version of \cite[Definition VI.5.1]{V1967}.

\begin{definition}\label{R-property}
An OVS satisfies the condition \text{\rm(R)} $($briefly, $X\in\text{\rm(R)}$$)$ if for each sequence $(y_k)$ in $X$ 
there exist $y\in X$ and a sequence $(\lambda_k)$ in $\mathbb{R}\setminus\{0\}$ such that $\pm\lambda_k y_k\le y$ for every $k$.
\end{definition}

\noindent
The condition \text{\rm(R)} is also called the $\sigma$-property \cite[Definition 70.1]{LZ1971}.
For examples of VLs satisfying the condition \text{\rm(R)} we refer to \cite{V1967,LZ1971,KK2020}.
The following lemma can be seen as an extension of \cite[Proposition 5.2]{KK2020} to the OVS setting.

\begin{lemma}\label{C-prop}
Let $X$ be an OVS with a generating cone. The following assertions are equivalent.
\begin{enumerate}[$i)$]
\item\
$X\in\text{\rm(R)}$.
\item\
For each sequence $(y_k)$ in $X_+$ there exist $y\in X$ and a sequence $(\lambda_k)$ in $\mathbb{R}_+\setminus\{0\}$ such that $\lambda_k y_k\le y$ for every $k$.
\item\
Any countable set of \text{\rm ru}-convergent nets in $X$ has a common regulator
\item\
Any countable set of \text{\rm ru}-convergent sequences in $X$ has a common regulator.
\end{enumerate}
\end{lemma}

\begin{proof}
$i)\Longrightarrow ii)$\ 
It is trivial.

\medskip
$ii)\Longrightarrow i)$\ 
Let $(y_k)$ be a sequence in $X$. Since $X_+-X_+=X$, there exist sequences $(z'_k)$ and $(z''_k)$ in $X_+$ satisfying
$y_k=z'_k-z''_k$. Take $z',z''\in X$ and sequences 
$(\lambda'_k)$ and $(\lambda''_k)$ in $\mathbb{R}_+\setminus\{0\}$ such that $\lambda'_k z'_k\le z'$ and $\lambda''_k z''_k\le z''$ for all $k$.
Let $y=z'+z''$ and $\lambda_k=\min(\lambda'_k,\lambda''_k)\in\mathbb{R}_+\setminus\{0\}$ for each $k$. It follows from
$$
   -z'-z''\le\min(\lambda'_k,\lambda''_k)(-z'_k-z''_k)\le\min(\lambda'_k,\lambda''_k)(z'_k-z''_k)=
$$
$$
   \lambda_ky_k\le\min(\lambda'_k,\lambda''_k)(z'_k+z''_k)\le z'+z''
$$
that $\pm\lambda_k y_k\le y$ for every $k$.

\medskip
$i)\Longrightarrow iii)$\ 
Let $x_{\alpha^{(k)},k}\convr x_k(y_k)$ for each $k\in\mathbb{N}$ as $\alpha^{(k)}\to\infty$.
Take $y\in X$ and a sequence $(\lambda_k)$ in $\mathbb{R}\setminus\{0\}$ with $\pm\lambda_k y_k\le y$ for all $k$.
Then $\pm y_k\le\frac{1}{|\lambda_k|}y$, and hence $x_{\alpha^{(k)},k}\convr x_k(y)$ as $\alpha^{(k)}\to\infty$ for each $k$.

\medskip
$iii)\Longrightarrow iv)$\ 
It is trivial.

\medskip
$iv)\Longrightarrow i)$\ 
Let $(y_k)$ be a sequence in $X$. Since $X_+-X_+=X$, there exists a sequence $(u_k)$ in $X_+$ with $\pm y_k\le u_k$. 
Since $n^{-1}u_k\convr 0(u_k)$ as $n\to\infty$ for each $k$, there exists $u\in X_+$ such that $n^{-1}u_k\convr 0(u)$
 for each $k$ as $n\to\infty$. Thus, for each $k$ there exists $n_k$ with $\pm n_k^{-1}u_k\le u$. Letting $\lambda_k=n_k^{-1}$
completes the proof.
\end{proof}

\medskip
In order to deal with localization at zero of the property to be \text{\rm ruc} in the class of positive semigroup on VLs, 
Kandi\'{c} and Kaplin introduce the property \text{\rm(D)} of VLs \cite[Definition 5.1]{KK2020}. We slightly modify
their definition by restricting it to positive semigroups rather than net and extent it to the OVS setting as follows.   

\medskip
\begin{definition}\label{Def.5.1-KK2020}
An OVS $X$ has the property \text{\rm(D)} $($briefly, $X\in\text{\rm(D)}$$)$ if for each positive semigroup $(T_s)_{s\ge 0}$ on $X$, 
$(T_s)_{s\ge 0}$ is \text{\rm ruc} whenever the following two conditions are satisfied.
\begin{enumerate}[$a)$]
\item\
There exists a \text{\rm ru}-dense set $D\subseteq X$ such that $T_hy\convr y$ as $h\downarrow 0$ for each $y\in D$.
\item\
If $x_n\convr 0$ then $\big(T_hx_n-x_n\big)_{(h,n)\in(0,\infty)\times\mathbb{N}}\convr 0$.
\end{enumerate}
\end{definition}

\noindent
\smallskip
It was shown in \cite{KK2020} that several important VLs (e.g., $\text{\rm L}^p(\mathbb{R})\ (0<p<\infty)$, $\text{\rm C}_c(\mathbb{R})$, 
$\text{\rm Lip}(\mathbb{R})$, $\text{\rm UC}(\mathbb{R})$, and $\text{\rm C}(\mathbb{R})$) have the property \text{\rm(D)}. 
The following proposition generalizes the implication $(ii)\Longrightarrow(iii)$ of \cite[Theorem 5.4]{KK2020} to the OVS setting.
The idea of our proof below is motivated by the idea the proof of Theorem 5.4 in \cite{KK2020}.

\smallskip
\begin{theorem}\label{R implies D}
Let  $X$ be an OVS with a generating cone. Then\\ $X\in\text{\rm(R)}\Longrightarrow X\in\text{\rm(D)}$. 
\end{theorem}

\smallskip
\begin{proof}
Let $X\in\text{\rm(R)}$ and suppose that a semigroup $(T_s)_{s\ge 0}$ on $X$ satisfies $a)$ and $b)$ of Definition \ref{Def.5.1-KK2020}.
In particular, $(T_hy-y)\convr 0$ as $h\downarrow 0$ for each $y$ belonging to an \text{\rm ru}-dense subset $D$ of $X$.
We need to prove that $T_{h}x\convr x$ as $h\downarrow 0$ for every $x\in X$. 

Let $x\in X$ and take
$(x_n)$ in $D$, $(x_n-x)\convr 0(u)$ as $n\to\infty$. Then, for each $k\in\mathbb{N}$ there exists $n'_k$ with 
$$
   \pm (x_n-x)\le\frac{1}{k}u \ \ \ \ \ \ \ (\forall n\ge n'_k).
   \eqno(3)
$$
By condition $b)$ of Definition \ref{Def.5.1-KK2020}, $\big(T_h(x_n-x)-(x_n-x)\big)_{(h,n)\in(0,\infty)\times\mathbb{N}}\convr 0$.
Since $(x_n-x)\convr 0$ then $\big(x_n-x\big)_{(h,n)\in(0,\infty)\times\mathbb{N}}\convr 0$, and hence
$\big(T_h(x_n-x)\big)_{(h,n)\in(0,\infty)\times\mathbb{N}}\convr 0$. Therefore, there exists $v\in X_+$ such that, 
for each $k\in\mathbb{N}$ there exist $n_k\ge n'_k$ and $h'_k>0$ with 
$$
   \pm T_h(x_n-x)\le\frac{1}{k}v \ \ \ \ \ \ \ (\forall n\ge n_k)(\forall h\in[0,h'_k]).
   \eqno(4)
$$
It follows from condition $a)$ of Definition \ref{Def.5.1-KK2020} that $T_hx_n\convr x_n$ as $h\downarrow 0$ for every $n$. 
Since $X\in\text{\rm(R)}$, Lemma \ref{C-prop}
implies that there exists $w\in X_+$ such that $(T_hx_n-x_n)\convr 0(w)$ as $h\downarrow 0$ for every $n$.
For each $k\in\mathbb{N}$ there exists $h''_k$, $0<h''_k\le h'_k$ with 
$$
   \pm(T_hx_{n_k}-x_{n_k})\le\frac{1}{k}w\ \ \ \ \ \ \ (\forall h\in[0,h''_k]).
   \eqno(5)
$$
Take any $h_k$, $0<h_k\le h''_k$. In view of (3), (4), and (5), for each $k\in\mathbb{N}$ 
$$
   -\frac{1}{k}(v+u+w)\le\big(T_h(x-x_{n_k})-(x-x_{n_k})\big)+\big(T_hx_{n_k}-x_{n_k}\big)=
$$
$$
   T_hx-x=\big(T_h(x-x_{n_k})-(x-x_{n_k})\big)+\big(T_hx_{n_k}-x_{n_k}\big)\le\frac{1}{k}(v+u+w)
$$
whenever $0\le h\le h_k$. Hence, $T_hx\convr x(v+u+w)$ as $h\downarrow 0$.
\end{proof}

\noindent
Many of \text{\rm ru}-complete VLs possessing the property \text{\rm(D)} have also the property \text{\rm(R)} \cite{KK2020}.
Accordingly to \cite[Theorems 2 and 3]{E2025}, every positive \text{\rm ruc} semigroup $(T_s)_{s\ge 0}$ on an Archimedean VL $X$ 
with the property \text{\rm(R)} has unique positive \text{\rm ruc} extension to the semigroup $(T^r_s)_{s\ge 0}$ on 
the \text{\rm ru}-completion $X^r$ of $X$. In the OVS setting \text{\rm ru}-completions are not available, so there is no reasonable hope
to drop \text{\rm ru}-completeness in attempts to extend results of \cite{KK-F2020} to OVSs. The following examples show that some of OVSs
commonly used in applications are not \text{\rm ru}-complete yet have the property \text{\rm(R)}.

\begin{example}\label{OBS with (R)}
Consider the OBS $X=C^n[0,1]$ of $n$-times continuously differentiable functions on $[0,1]$.
Since $C^n[0,1]$ is not complete in the sup-norm, $X$ is not \text{\rm ru}-complete. From the other hand, 
since $1_{[0,1]}$ is an order unit of $C^n[0,1]$ then $X\in\text{\rm(R)}$.
By the same reason, $X_+$ it generating. 
\end{example}

\begin{example}\label{OVS with (R)}
Let $X=C^\infty(\mathbb{R}^k,\mathbb{R})$ be an OVS of infinitely many times differentiable $\mathbb{R}$-valued functions of $k$ varables.
It can be easily seen that $X$ it is not \text{\rm ru}-complete and satisfies the condition \text{\rm(R)}.
\end{example}

\medskip
Next, we give the following localization at zero of \text{\rm ruc}-property that generalizes \cite[Theorem 5.7]{KK2020} to the OVS setting.

\smallskip
\begin{theorem}\label{Thm.5.7-KK2020}
Let $X$ be an OVS with a generating cone such that $X\in\text{\rm(D)}$. 
Then a positive semigroup $(T_s)_{s\ge 0}$ on $X$ is \text{\rm ruc} if and only if the following two assertions hold.
\begin{enumerate}[$i)$]
\item\
There exists an \text{\rm ru}-dense set $D\subseteq X$ such that $T_hy\convr y$ as $h\downarrow 0$ for each $y\in D$.
\item\
The semigroup $(T_s)_{s\ge 0}$ is order bounded at zero.
\end{enumerate}
\end{theorem}

\begin{proof}
For the necessity, take $D=X$ in $i)$ and apply Theorem~\ref{ob-ruc}~$i)$.
 
\medskip
For the sufficiency, suppose $i)$ and $ii)$ hold. Since the assertion $i)$ coincides with the condition $a)$ of Definition \ref{Def.5.1-KK2020},
it is enough to check condition $b)$ of this definition. Let $x_n\convr 0$ in $X$, say $x_n\convr 0(y)$.
So, there exists an increasing sequence $(k_n)$ with $\pm x_n\le\frac{1}{n}y$ for all $k\ge k_n$.
It follows from $ii)$ that there exist $z\in X$ and $h_0>0$ such that $0\le T_hy\le z$ and all $h\in[0,h_0]$. Therefore,
for all $n\ge n_\varepsilon$ and $h\in[0,h_0]$ we have
$$
   -\frac{1}{n}\Big(z+y\Big)\le-\frac{1}{n}T_hy-\frac{1}{n}y\le T_hx_n-x_n\le\frac{1}{n}T_hy+\frac{1}{n}y\le\frac{1}{n}\Big(z+y\Big),
$$
and hence $(T_hx_n-x_n)_{(h,n)\in(0,\infty)\times \mathbb{N}}\convr 0(z+y)$.
\end{proof}

\section{\large Conditions under which a positive $C_0$-semigroup on an OBS is ruc}

\hspace{4mm}
Gl\"{u}ck and Kaplin proved that a positive $C_0$-semigroup on a BL is \text{\rm ruc} 
if and only if it is order bounded at zero \cite[Theorem 2.1]{GK2024}. It is worth to mention that the condition of \cite[Theorem 2.1]{GK2024} 
on a semigroup to be $C_0$ is essential (for instance, the right translation semigroup on $L^{\infty}({\mathbb R})$ 
is not a strongly continuous at zero, and hence is not \text{\rm ruc} at zero, yet it is order bounded at zero).

\medskip
In order to extend \cite[Theorem 2.1]{GK2024} to the OBS setting, we need several lemmas. 
All integrals below are understood in the Bochner sense.

\begin{lemma}\label{localization at zero 1}
Let $(T_s)_{s\ge 0}$ be a positive \text{\rm ruc} at zero semigroup on an almost Archime\-dean OVS $X$.
Then $(T_s)_{s\ge 0}$ is order continuous at zero.
\end{lemma}

\begin{proof}
Let $x\in X$. There exist $u\in X$ and a sequence $h_n\downarrow 0$ satisfying  
$-\frac{1}{n}u\le T_hx-x\le\frac{1}{n}u$ for all $n\in\mathbb{N}$ and $h\in[0,h_n]$. 
Since $X$ is almost Archime\-dean, $\frac{1}{n}u\downarrow 0$, and hence $T_hx\convo x$ as $h\downarrow 0$.
\end{proof}

\begin{lemma}\label{normtoru}
Let $(x_n)$ be a sequence in an OBS $X$ with a closed generating cone and let $x\in X$. The following conditions are equivalent.
\begin{enumerate}[$i)$]
\item\
$x_n\convn x$.
\item\
Each subsequence $(x_{n_k})$ of $(x_n)$ contains further subsequence $(x_{n_{k_j}})$ with $x_{n_{k_j}}\convr x$.
\end{enumerate}
\end{lemma}

\begin{proof}
It suffices to consider the case when $x=0$. 

\medskip
$i)\Longrightarrow ii)$\ 
Let $x_n\convn 0$ and let $(x_{n_k})$ be a subsequence of $(x_n)$. Then $x_{n_k}\convn 0$.
By the M. Krein -- V. \v{S}mulian theorem (cf., \cite[Theorem 2.37]{AT2007}),
$\alpha B_X\subseteq B_X\cap X_+-B_X\cap X_+$ for some $\alpha>0$, and hence, for each $x_n$ there exist 
$v_n,w_n\in X_+$, $x_n=v_n-w_n$ and $\max(\|v_n\|,\|w_n\|)\le\alpha^{-1}\|x_n\|$. 
For each $j\in\mathbb{N}$ find $n_{k_j}$ with $\|v_{n_{k_j}}+w_{n_{k_j}}\|\le\frac{1}{j^3}$,
and consider $u=\|\cdot\|-\sum_{j=1}^\infty j(v_{n_{k_j}}+w_{n_{k_j}})$. Then, 
$\pm x_{n_{k_j}}\le v_{n_{k_j}}+w_{n_{k_j}}\le\frac{1}{j}u$ for each $j$, and hence $x_{n_{k_j}}\convr 0(u)$.

\medskip
$ii)\Longrightarrow i)$\ 
Let in contrary $\|x_n\|\not\to 0$. Then, for some $\epsilon > 0$, there exists a subsequence $(x_{n_k})$ 
such that $\|x_{n_k}\|\ge\epsilon$ for all $k$. By the assumption, $x_{n_{k_j}}\convr 0$ for some 
further subsequence $(x_{n_{k_j}})$, and hence $x_{n_{k_j}}\convn 0$, which is absurd.
\end{proof}

The following lemma is an OBS version of \cite[Lemma 2.2]{GK2024}.

\begin{lemma}\label{Lemma 2.2-OBS}
Let $(T_s)_{s\ge 0}$ be a positive $C_0$-semigroup with generator $(A,D(A))$ on an OBS $X$ with a closed 
generating normal cone, and let $\omega$ be the growth bound of $(T_s)_{s\ge 0}$. Then the following holds.
\begin{enumerate}[$i)$]
\item\
For each $g\in D(A)$ there exists $z\in D(A)$ such that\\ $\pm T_sg\le e^{\omega s}z$ for all $s\ge 0$.
\item\
For each $y\in D(A)$, $T_sy\convr y(q)$ with regulator $q\in D(A)$ as $s\downarrow 0$.
\end{enumerate}
\end{lemma}

\begin{proof}
$i)$\
Let $g\in D(A)$. Since $X_+-X_+=X$, there exists $b\in X_+$ such that $b\ge\pm(\omega I-A)g$. Set
$z:=R(\omega,A)b=\int_0^\infty e^{-\omega t}T_tbdt\in D(A)$. Then, for every $s\ge 0$
$$
   \pm T_sg=\pm R(\omega,A)T_s(\omega I-A)g=\pm\int_0^\infty e^{-\omega t}T_{t+s}(\omega I-A)gdt=
$$
$$
   \pm\int_s^\infty e^{-\omega t}T_{t+s}(\omega I-A)gdt\le\int_s^\infty e^{-\omega t}T_{t+s}bdt=
$$
$$
   e^{\omega s}\int_s^\infty e^{-\omega(t+s)}T_{t+s}bdt\le e^{\omega s}R(\omega,A)b=e^{\omega s}z.
$$

$ii)$\
Let $y\in D(A)$ and take a sequence $(x_n)$ in $D(A)$ that converges to $(\omega I-A)y$ in the norm.  
In view of Lemma \ref{normtoru}, we may assume $\pm(x_n-(\omega I-A)y)\le\frac{1}{n}u$ for all $n$ and some $u\in X_+$. 
Denote $y_n:=R(\omega,A)x_n \in D(A^2)$. Since $R(\omega, A)\ge 0$ then, for each $n$, 
$$
   \pm R(\omega,A)(x_n-(\omega I-A)y)=\pm(y_n-y)\le\frac{1}{n}R(\omega,A)u.
$$
Since $X_+-X_+=X$, there exist $w_n,v_n\in X_+$ such that 
$$
   \pm Ay_n\le w_n \ \ \ \& \ \ \ \pm(\omega I-A)w_n\le v_n
   \eqno(6)
$$ 
for all $n$. Since $(T_s)_{s \geq0}$ is positive, (6) implies
$$
  T_s w_n=T_sR(\omega,A)(\omega I-A)w_n=R(\omega,A)T_s(\omega I-A)w_n=
$$
$$
  \int_0^\infty  e^{-\omega t}T_{t+s}(\omega I-A)w_ndt=
  e^{\omega s}\int_0^\infty  e^{-\omega(t+s)}T_{t+s}(\omega I-A)w_ndt=
$$
$$
  e^{\omega s}\int_s^\infty  e^{-\omega t}T_t(\omega I-A)w_ndt\le
  e^{\omega s}\int_s^\infty  e^{-\omega t}T_tv_ndt\le  
$$
$$
   e^{\omega s}\int_0^\infty  e^{-\omega t}T_tv_ndt=
   e^{\omega s}R(\omega,A)v_n.
  \eqno(7)
$$ 
It follows from (6) and (7) that 
$$
   \pm(T_sy_n-y_n)=\pm\int_0^s T_tAy_ndt\le\int_0^sT_tw_nds\le s e^{\omega s}R(\omega,A)v_n.
  \eqno(8)
$$
Take $v:=\|\cdot \|-\sum_{n=1}^\infty\frac{v_n}{2^n (\|v_n \| +1)}$. It follows from (8) that, for each $n$, 
$T_sy_n\convr y_n$ with regulator $R(\omega,A)v$ as $s\downarrow 0$.

\medskip
To show $T_sy\convr y$, let $\varepsilon>0$ and choose $n_0$ such that $\frac{1}{n_0}\le\varepsilon$.
First, we apply $i)$ to $g=R(\omega,A)u$ and find that $T_sR(\omega,A)u\le e^{\omega s}z$ for some $z\in X_+$ and all $s\ge 0$.
It follows from (8) that 
$$
   \pm(T_sy_{n_0}-y_{n_0})\le s e^{\omega s}R(\omega,A)v_{n_0}\le 2^n (\|v_n \| +1)s e^{\omega s}R(\omega,A)v,
$$ 
and hence there exists $s_0>0$ such that $\pm(T_sy_{n_0}-y_{n_0})\le\varepsilon R(\omega,A)v$ for all $s\in (0,s_0]$. Then
$$
   \pm(T_sy-y)=\pm\big((T_sy-T_sy_{n_0})+(T_sy_{n_0}-y_{n_0})+(y_{n_0}-y)\big)\le
$$ 
$$
   \frac{1}{n_0}T_sR(\omega,A)u+\varepsilon R(\omega,A)v+\frac{1}{n_0}R(\omega,A)u\le
$$
$$
   \varepsilon(e^{\omega s}z+R(\omega,A)v+R(\omega,A)u)\le\varepsilon(e^{|\omega|s_0}z+R(\omega,A)v+R(\omega,A)u).
$$
Thus, $T_sy \convr y(q)$ as $s\downarrow 0$ with the regulator $q=e^{|\omega|s_0}z+R(\omega,A)v+R(\omega,A)u\in D(A)$.
\end{proof}

\begin{lemma}\label{ob0 to ruc0}
Let $(T_s)_{s\ge 0}$ be a positive order bounded at zero $C_0$-semigroup with generator $(A,D(A))$ on an OBS $X$ with a closed 
generating and normal cone. Then $(T_s)_{s\ge 0}$ is ruc at zero.
\end{lemma}

\begin{proof}
Let $x\in X$. Take a sequence $(x_n)$ in $D(A)$, $\|x_n-x\|\to 0$. 
By Lemma \ref{normtoru}, we may suppose $x_n\convr x(u_0)$.
Lemma \ref{Lemma 2.2-OBS}~$ii)$ implies $T_hx_n\convr x_n(u_n)$ for each $n\in\mathbb{N}$ as $h\downarrow 0$. 
Take $w:=\|\cdot \|-\sum_{n=0}^\infty\frac{v_n}{2^n (\|v_n \| +1)}$. Then $x_n\convr x(w)$ and $T_hx_n\convr x_n(w)$ as $h\downarrow 0$.
Since $(T_s)_{s\ge 0}$ is order bounded at zero, there exist $s_0>0$ and $v\in X$ such that 
$T_hw\le v$ for all $h\in[0,s_0]$. 

Let $\varepsilon>0$. Choose $n_0$ so that $\pm(x_{n_0}-x)\le\varepsilon w$ and then choose $s_1>0$
with $\pm(T_hx_{n_0}-x_{n_0})\le\varepsilon w$ for all $h\in[0,s_1]$. Then
$$
   \pm(T_hx-x)=\pm\big(T_h(x-x_{n_0})+(T_hx_{n_0}-x_{n_0})+(x_{n_0}-x)\big)\le
$$
$$
   T_h(\varepsilon w)+\varepsilon w+\varepsilon w\le\varepsilon(v+2w)
$$
for $0\le h\le\min(s_1,s_2)$, which proves $T_hx\convr x$ as $h\downarrow 0$.
\end{proof}

\medskip
Now, we are in the position to prove the following generalization of \cite[Theorem 2.1]{GK2024} to the OBS setting. 

\begin{theorem}\label{localization at zero of onbd in OBS}
Let $(T_s)_{s\ge 0}$ be a positive $C_0$-semigroup on an ordered Banach space $X$ with a closed generating normal cone. 
The following assertions are equivalent.
\begin{enumerate}[$i)$]
\item
$(T_s)_{s\ge 0}$ is \text{\rm ruc}.
\item
$(T_s)_{s\ge 0}$ is \text{\rm ruc} at zero.
\item
$(T_s)_{s\ge 0}$ is order continuous at zero.
\item
$(T_s)_{s\ge 0}$ is order bounded at zero.
\item
For every $x\in X$ and $M\ge 0$ the set $\{T_hx\}_{0\le h\le M}$ is order bounded.
\item
For every $x\in X_+$ and $M\ge 0$ the set $\bigcup\limits_{0\le h\le M}T_h[-x,x]$ is order bounded. 
\end{enumerate}
\end{theorem}

\begin{proof}
$i)\Longrightarrow ii)\Longrightarrow iii)$\ 
The normality of $X_+$ ensures that $X$ is almost Archime\-dean. The rest follows from Lemma \ref{localization at zero 1}

\medskip
$iii)\Longrightarrow iv)$\ 
Let $x\in X$. Since $T_hx\convo x$ as $h\downarrow 0$, there exists a net $g_\beta\downarrow 0$ in $X$ such that,
for every $\beta$ there exists $\varepsilon_\beta>0$ with $\pm(T_hx-x)\le g_\beta$ for all $h\in[0,\varepsilon_\beta]$. 
Fix some $\beta_0$ and take $w\ge\pm x$.
Then $\pm T_hx\le g_{\beta_0}+w$ for all $h\in[0,\varepsilon_{\beta_0}]$, and hence
$\big\{T_hx\big\}_{0\le h\le\varepsilon_{\beta_0}}\subseteq[-(g_{\beta_0}+w),g_{\beta_0}+w]$.

\medskip
$iv)\Longleftrightarrow v)\Longleftrightarrow vi)$\ 
It follows from Proposition \ref{localization at zero of obz in OVS}.

\medskip
$iv)\Longrightarrow i)$\ 
It is Lemma \ref{ob0 to ruc0}.
\end{proof}

\noindent
Example \ref{example -(n+1)} shows that, without the order boundedness assumption on operators $T_s$ in Theorem \ref{localization at zero of onbd in OBS},
its conclusion is not true even when $X$ is an order continuous BL. We have no example of a $C_0$-semigroup $(T_s)_{s\ge 0}$ of order bounded operators 
on an OVS with a closed generating normal cone, for which the conclusion of Theorem \ref{localization at zero of onbd in OBS} fails.

%
%

{}
\end{document}